\documentclass[12pt,a4paper]{article}

\usepackage[utf8]{inputenc}
\usepackage[T1]{fontenc}

\usepackage{amsmath}
\usepackage{amsthm}
\usepackage{amssymb}
\usepackage{amsfonts}
\usepackage{dsfont}
\usepackage{url}
\usepackage{graphicx}
\usepackage{hyperref}
\usepackage{color}
\usepackage{tikz}
\usepackage{pgfplots}
\usetikzlibrary{shapes.misc}
\usepackage{enumerate}
\usepackage{fullpage}

\theoremstyle{plain}
\newtheorem{theorem}{Theorem}[section]
\newtheorem{lemma}[theorem]{Lemma}
\newtheorem*{remark}{Remark}
\newtheorem{as}{Assumption}

\pgfplotsset{compat=1.13}
\renewcommand{\P}{\mathbb{P}} 
\newcommand{\E}{\mathbb{E}} 
\newcommand{\N}{\mathbb{N}} 
\newcommand{\R}{\mathbb{R}} 
\newcommand{\limn}{\lim_{n \to \infty}}

\newcommand{\Addresses}{{%
  \bigskip
  \footnotesize

  N.~Gantert, \textsc{Fakultät für Mathematik, Technische Universität München, 
  Boltzmannstr.~3, 85748 Garching, Germany}\par\nopagebreak
  \textit{E-mail address:} \texttt{gantert@ma.tum.de}
  
  \medskip
  
  T.~Höfelsauer, \textsc{Fakultät für Mathematik, Technische Universität München, 
  Boltzmannstr.~3, 85748 Garching, Germany}\par\nopagebreak
  \textit{E-mail address:} \texttt{thomas.hoefelsauer@tum.de}

}}

\begin{document}
\title{Large deviations for the maximum of a branching random walk}
\author{Nina Gantert and Thomas H\"ofelsauer}
\maketitle

\begin{abstract}
We consider real-valued branching random walks and prove a large deviation result for the position of the rightmost particle. The position of the rightmost particle is the maximum of a collection of a random number of dependent random walks. We characterise the rate function as the solution of a variational problem.
We consider the same random number of independent random walks, and show that the maximum of the branching random walk is dominated by the maximum of the independent random walks. 
For the maximum of independent random walks, we derive a large deviation principle as well.
It turns out that the rate functions for upper large deviations coincide, but in general the rate functions for lower large deviations do not.
 

\end{abstract}

\section{Introduction} \label{introduction}
We study the maximum of a branching random walk in discrete time. The branching random walk can be described as follows.
Given are two ingredients, an offspring distribution on $\N_0$ with weights $(p(k))_{k \in \N_0}$ and a (centred) step size distribution on $\R$.
At time $0$ the process starts with one particle at the origin. At every time $n \in \N$ each particle repeats the following two steps, independently of everything else. First, it produces offspring according to the offspring distribution $(p(k))_{k \in \N_0}$, and then it dies. Afterwards, the offspring particles move according to the step size distribution. Let $m$ be the reproduction mean of the offspring distribution. We always assume $m > 1$.
We are interested in the position of the rightmost particle at time $n$. Therefore, let $D_n$ be the set of particles at time $n$. For $v \in D_n$ denote by $S_v$ the position of particle $v$ at time $n$. The maximum of the branching random walk is defined as
\begin{equation}\label{Mndef}
M_n= \max_{v \in D_n} S_v.
\end{equation}
The asymptotics of the maximum $M_n$ are by now well-understood. Let $I$ be the rate function of the random walk. Conditioned on survival, it is well-known, see Biggins \cite{B76}, Hammersley \cite{H74} and Kingman \cite{K75}, that $M_n$ grows at linear speed $x^*$, where
\begin{equation} \label{def_speed}
x^*=\sup \bigl\{ x \geq 0 \colon I(x) \leq \log m \bigr\}.
\end{equation}
Addario-Berry and Reed \cite{ABR09} as well as Hu and Shi \cite{HS09} obtain a logarithmic second term. In \cite{A13} A{\"i}d{\'e}kon finally proves that $M_n - x^*n + c \log n$ converges in distribution, where $c>0$ is an explicit constant. Furthermore, let $\tilde{M}_n$ be the maximum of $|D_n|$ independent random walks. One can check that the speed of $(\tilde{M}_n)_n$ also equals $x^*$, see \cite[Theorem 1]{Z16} or \eqref{proof_thm_ind_RW_4}. However, compared to the branching random walk, there is a larger logarithmic correction term, see e.g.~\cite[Theorem 1]{Z16}.
Similar results were proved for branching Brownian motion by Bramson in \cite{B78}.

We investigate the exponential decay rates of the probabilities $\P \bigl( \frac{M_n}{n} \geq x \bigr)$ for $x \geq x^*$ and
$\P \bigl( \frac{M_n}{n} \leq x \bigr)$ for $x \leq x^*$. Our main result, Theorem \ref{theorem_BRW}, characterises these exponential decay rates. We consider the same question for $\tilde{M}_n$ and determine the exponential decay rates, see Theorem \ref{theorem_ind_RW}.
Interestingly, the rate functions coincide for  $x \geq x^*$, but in general they do not coincide for $x < x^*$.
  
Similar questions have been studied before.
Large deviation estimates for the maximum of branching Brownian motion have first been investigated by  Chauvin and Rouault in \cite{CR88} and very recently by 
Derrida and Shi in \cite{DS17} and \cite{DS17_2}. See also
\cite{DS16} and  \cite{S13} for extensions with coalescence and  selection or immigration, respectively.
Note that \cite{DS17_2}
also treats continuous time branching random walks. The difference to our setup is that in the time-continuous case, 
the strategies can involve the exponential waiting times, while in our setup, they can involve the branching mechanism given by the offspring distribution.

Upper large deviations for the maximum of discrete time branching random walks have been investigated in \cite{R93} in the case where every particle has at least one offspring. Recently large deviation results for the empirical distribution of the branching random walk have been obtained in \cite{LP15}, \cite{CH17}, \cite{LT17}, but they do not seem to imply our result. We also mention that in the case of a fixed number of offspring, much more precise results (describing not only the exponential decay rates) were derived in \cite{BM17}.

Our strategy of proof is rather direct. 
We compare $M_n$ and $\tilde{M}_n$ and show that $\tilde{M}_n$ stochastically dominates $M_n$ for all $n \in \N$, see Lemma~\ref{lemma_BRW_vs_ind_RW2}.

Let us now introduce the model in a more formal way and fix some notation. Let $(Z_n)_{n \in \N_0}$ be a Galton-Watson process with one initial particle and offspring law given by the weights $(p(k))_{k \in \N_0}$. Let $m=\sum_{k=1}^\infty k p(k) $ be the reproduction mean. 
The associated Galton-Watson tree is denoted by $\mathcal{T}=(V,E)$, where $V$ is the set of vertices and $E$ is the set of edges. Further, for $n \in \N$ let $D_n$ be the set of vertices in the $n$-th generation of the tree. Then, $|D_n|=Z_n$. For $v \in D_n$ the set of descendants of $v$ in the $(l+n)$-th generation is denoted by $D_l^v$. Note that $|D_l^v|$ equals $|D_l|$ in distribution. The root of $\mathcal{T}$ is called $o \in V$. For $v,w \in V$ define $[v,w]$ as the set of edges along the unique path from $v$ to $w$. We now define the locations of the particles. Let $(X_e)_{e \in E}$ be a collection of i.i.d.~random variables, i.e.~every edge of $\mathcal{T}$ is labelled with a random variable. For $v \in D_n$ the position of the particle $v$ at time $n$ is defined as $S_v=\sum_{w \in [o,v]} X_w$. For $n \in \N$ the position of the rightmost particle at time $n$ is $M_n= \max_{v \in D_n} S_v$ and if $D_n=\emptyset$, we set $M_n = - \infty$. We refer to $(M_n)_{n \in \N}$ as the maximum of the branching random walk. For $v \in D_n$, the rightmost descendant of $v$ at time $l+n$ is defined as $M_l^v=\max_{w \in D_l^v} S_w$. 

We also introduce a collection of i.i.d.~random variables $(X_i^j)_{i,j \in \N}$, where $X_1^1$ has the same distribution as $X_e$ for some $e \in E$. Moreover, for $j,n \in \N$ define the random walk $S_n^j= \sum_{i=1}^n X_i^j$ and the maximum of independent random walks as
\begin{equation}\label{Mntildedef}
\tilde{M}_n = \max_{1 \leq j \leq Z_n} S_n^j.
\end{equation}
In analogy to the maximum of the branching random walk, we set $\tilde{M}_n=- \infty$, if $D_n = \emptyset$. Furthermore, for $i \in \N$ let $X_i$ be an independent copy of $X_i^1$ and define $S_n= \sum_{i=1}^n X_i$.
Note that for every time $n$ the number of particles in the branching random equals the number of random walks considered for $\tilde{M}_n$. However, the positions of the particles in the branching random walk are not independent. Indeed, this dependence is such that the maximum of independent random walks stochastically dominates the maximum of the branching random walk, see Lemma~\ref{lemma_BRW_vs_ind_RW2}. 
We introduce the measure
\begin{equation} \label{def_P*}
\P^*( \cdot) = \P( \cdot \vert Z_n > 0 \ \forall n\in \N).
\end{equation}
The associated expectation is denoted by $\E^*$.
Let $(a_n)_{n \in \N}$ be a sequence of positive numbers and let $c \in (0,\infty]$ be a constant. With a slight abuse of notation for $c=\infty$, we write $a_n=\exp(-cn + o(n))$, if $ \limn \frac{1}{n} \log a_n = -c$. Note that $a_n$ decays faster than exponentially in $n$ if $c = \infty$.
The stochastic processes considered in this paper are discrete time processes. However, to increase the readability of the paper, we omit integer parts if no confusion arises.

The paper is organised as follows. In Section~2 we first introduce the rate function of the random walk and two rate functions concerning the Galton-Watson process. We further describe our assumptions. Then we state our main results in Section~3. We collect some preliminary results in Section~4 and prove the main results in Section~5.

\section{Rate functions and assumptions} \label{ratefcts}
In this section we introduce the rate functions of the random walk and the Galton-Watson process, which are needed to state our main results.
\subsection{Rate function of the random walk}
For $x \in \R$ the rate function of the random walk $(S_n)_{n \in \N}$ is defined as

\begin{equation} \label{def_I}
I(x)= \sup_{\lambda \in \R} \bigl(\lambda x - \log \E\left[e^{\lambda X_1} \right]\bigr).
\end{equation}

\begin{as} \label{assumption_RW}
There exists $\varepsilon>0$ such that $\E\bigl[ e^{\lambda X_1} \bigr] < \infty$ for all $\lambda \in (-\varepsilon, \varepsilon)$. Furthermore, for simplicity suppose that $\E[X_1]=0$.
\end{as}
Note that $\E[X_1]=0$ is not necessary for the results in Section~\ref{results}, since we could derive similar results for the collection of random variables $(X_e-\E[X_1])_{e \in E}$.
Assumption \ref{assumption_RW} ensures that $I(x)>0$ for all $x \neq 0$ and $I(x) \to \infty$ as $|x| \to \infty$. 
If Assumption \ref{assumption_RW} is satisfied, Cramér's theorem implies that the 
probabilities $\P(S_n \geq xn)$ decay exponentially in $n$ with rate $I(x)$ for $x>0$.
A proof can e.g.~be found in \cite[Theorem 3.7.4]{DZ10}.

\subsection{Rate functions of the Galton-Watson process}

First, we need to introduce some more assumptions before we can state the large deviation results.
\begin{as} \label{assumption_supercritical}
The Galton-Watson process is supercritical, i.e.~$m>1$.
\end{as}
Let $ q=\inf \bigl\{s \in [0,1] \colon \E[s^{Z_1}]=s \bigr\}$.
Note that $q$ is the extinction probability of the process $(Z_n)_{n \in \N_0}$. Assumption \ref{assumption_supercritical} implies that $q < 1$.
\begin{as} \label{assumption_ZlogZ}
The Galton-Watson process satisfies $\E[Z_1 \log Z_1] < \infty$.
\end{as}
First of all, note that Assumption \ref{assumption_ZlogZ} implies $m < \infty$. Together with Assumption \ref{assumption_RW}, this yields that the speed of the branching random walk $x^*$ defined in \eqref{def_speed} is finite. 

Due to the well-known Kesten-Stigum Theorem,  Assumption \ref{assumption_ZlogZ} implies that the Galton-Watson process grows like its expectation, see Theorem~\ref{Lemma_GWP_martingale}.

\begin{as} \label{assumption_Schroeder}
Every particle has less than two children with positive probability, i.e.~it holds that $p(0)+p(1)>0$.
\end{as}
Assumption \ref{assumption_Schroeder} is often referred to as Schröder case, whereas the case $p(0)+ p(1)=0$ is called Böttcher case.

If Assumptions \ref{assumption_supercritical} and \ref{assumption_ZlogZ} are satisfied, there is a large deviation result for the probability that $(Z_n)_{n \in \N_0}$ grows at most subexponentially. A sequence $(a_n)_{n \in \N}$ is called subexponential, if $a_n e^{-\varepsilon n} \to 0$ as $n \to \infty$ for all $\varepsilon>0$. 
Define
\begin{equation}\label{def_rho}
\rho: = -\log \E[Z_1 q^{Z_1-1}] \in (0, \infty].
\end{equation}
Note that $\rho=-\log p(1)$ if $p(0)=0$ (and therefore also $q=0$). In particular, $\rho <\infty$ if and only if Assumption \ref{assumption_Schroeder} is satisfied. Consider the set
\begin{equation} \label{def_A}
A= \bigl\{l \in \N \colon \exists n \in \N \text{ such that } \P(Z_n=l)>0 \bigr\}
\end{equation}
containing all positive integers $l$ such that there are $l$ particles at some time $n$ with positive probability.
\begin{theorem}\label{lemma_LDP_rho}
Let Assumption \ref{assumption_supercritical} and \ref{assumption_ZlogZ} hold. Then, for every $k \in A$ we have 
\begin{equation*}
\limn \frac{1}{n} \log \P^* \bigl(Z_n =k \bigr) = - \rho.
\end{equation*} 
Moreover, for every subexponential sequence $(a_n)_{n \in \N}$ such that $a_n \to \infty$ as $n \to \infty$,
\begin{equation*}
\limn \frac{1}{n} \log \P^* \bigl(Z_n \leq a_n \bigr) = - \rho.
\end{equation*}
\end{theorem}
A proof of the first statement in Theorem~\ref{lemma_LDP_rho} can be found in \cite[Chapter 1, Section 11, Theorem 3]{AN04}. The second statement is a consequence of of \cite[Theorem 3.1]{BB13}.

For $x \in [0, \log m]$ define the rate function of the Galton-Watson process as
\begin{equation} \label{def_I_GW}
I^{\text{GW}}(x)= \rho \bigl(1-x (\log m)^{-1} \bigr) 
\end{equation} 
Note that $I^{\text{GW}}(x)>0$ for all $x < \log m$.
\begin{theorem}\label{lemma_LDP_Zn}
Under Assumptions \ref{assumption_supercritical} and \ref{assumption_ZlogZ} we have for $x \in [0, \log m]$
\begin{equation*}
\limn \frac{1}{n} \log \P^* \bigl( Z_n \leq e^{xn} \bigr) = -  I^{\textnormal{GW}}(x).
\end{equation*}
\end{theorem}
This theorem is a consequence of \cite[Theorem 3.2]{BB13}. Note that there is also an upper large deviation result for $\P^* \bigl( Z_n \geq e^{xn} \bigr)$, where $x > \log m$, see e.g.~\cite[Theorem 1]{BB11}.

\section{Results} \label{results}
After defining the rate functions of the random walk and the Galton-Watson process we are now able to state our main results.

Note that $I(x^*)= \log m$ if $I(x)< \infty$ for some $x > x^*$. On the other hand, $I(x^*)< \log m$ already implies $\P(X_1>x^*)=0$. This case leads to a different shape of the rate functions, see Figure \ref{fig_rate_functions}.
Let
\begin{equation} \label{def_k*}
k^*=\inf\{k \geq 1 \colon p(k)>0 \}.
\end{equation}
Note that $k=k^*$ is the smallest positive integer, such that $\P(Z_n=k)>0$ for some $n \in \N$. Define the rate function for the maximum of independent random walks as
\begin{equation} \label{def_I_ind}
I^{\text{ind}}(x)=
\begin{cases}
I(x) - \log m & \text{for } x > x^*,\\
0 & \text{for } x=x^*,\\
\rho \bigl(1 - \tfrac{I(x)}{\log m} \bigr)  &  \text{for } 0 \leq x < x^*,\\
k^* I(x) + \rho & \text{for } x \leq 0.
\end{cases}
\end{equation}
Note that $\rho \bigl(1 - \tfrac{I(x)}{\log m} \bigr)= I^{\text{GW}}(I(x))$ for $ 0 \leq x < x^*$.
Recall the maximum $\tilde{M}_n$ of a random number of independent walks, defined in \eqref{Mntildedef}.

\begin{theorem}\label{theorem_ind_RW}
Suppose that Assumptions \ref{assumption_RW}, \ref{assumption_supercritical} and \ref{assumption_ZlogZ} are satisfied. 
 Then, the laws of $\frac{\tilde{M}_n}{n}$ under $\P^*$ satisfy a large deviation principle with rate function $I^{\text{ind}}$. More precisely,
\begin{align*}
- I^{\textnormal{ind}}(x)=
\begin{cases}
\limn \frac{1}{n} \log \P \bigl( \frac{\tilde{M}_n}{n} \geq x \bigr) & \text{for } x \geq x^*,\\
\limn \frac{1}{n} \log \P \bigl( \frac{\tilde{M}_n}{n} \leq x \bigr) & \text{for } x \leq x^*.
\end{cases}
\end{align*}
\end{theorem}

In the Böttcher case ($p(0)+p(1)=0$) we have $\rho=\infty$ and therefore $I^{\text{ind}}(x)=\infty$ for all $x<x^*$. Hence, in this case the lower deviation probabilities $\P^*( \tilde{M}_n \leq xn)$ for $x<x^*$ decay faster than exponentially in $n$.

Let us now give some intuition for the rate function $I^{\text{ind}}$ and describe the large deviation event $\{\tilde{M}_n \geq xn\}$ for some $x>x^*$, respectively $\{\tilde{M}_n \leq xn\}$ for some $x<x^*$.

For $x>x^*$, the number of particles should be larger or equal than expected, i.e.~$Z_n \geq e^{nt}$ for some $t \geq \log m$. The probability of such an event is of order $\exp(-I^{\text{GW}}(t)n + o(n))$. If there are $e^{nt}$ particles at time $n$, the probability that at least one particle reaches $xn$ is of order $\exp(-I(x)n+tn+o(n))$ for $t < I(x)$. Therefore, we need to maximize the product of these two probabilities, which amounts to minimize $I^{\text{GW}}(t)+ I(x) - t$, where $t$ runs over the interval $[\log m, I(x))$.
It turns out that the optimal value is $t=\log m$. This argument will go through for the maximum of the branching random walk.

If $0 \leq x< x^*$, the probability that one particle reaches $xn$ is of order $\exp(-I(x)n+o(n))$. Hence, for every $\varepsilon>0$, if there are less than $e^{(I(x)- \varepsilon) n}$ particles, the probability that none of these particles reaches $xn$ is close to 1. However, if there are more than $e^{(I(x)+ \varepsilon) n}$ particles, this probability decays exponentially in $n$.

If $x<0$, already the probability that a single particle is below $xn$ at time $n$ decays exponentially fast in $n$. Hence, if the number of particles $Z_n$ grows exponentially, the probability that all particles are below $xn$ at time $n$ decays faster than exponentially. Therefore, the number of particles needs to grow subexponentially. Since $\rho$ does not depend on the choice of $k$ in Theorem~\ref{lemma_LDP_rho}, there have to be only $k^*$ particles at time $n$ (provided that $\rho < \infty$).

Next, we consider the maximum of the branching random walk. For $x < x^*$ let
\begin{equation} \label{def_H}
H(x)=\inf_{t \in (0,1]} \Bigl\{ t \rho + t I \bigl(t^{-1} \bigl(x-(1-t)x^* \bigr) \bigr) \Bigr\} .
\end{equation}
Note that for $x > 0$ it suffices to take the infimum over $t \in (0,1-\frac{x}{x^*}]$. 
Define the rate function of the branching random walk as
\begin{equation} \label{def_I_BRW}
I^{\text{BRW}}(x)=
\begin{cases}
I(x) - \log m & \text{for } x > x^*,\\
0 & \text{for } x=x^*,\\
H(x) & \text{for } x < x^*.
\end{cases}
\end{equation}

\begin{theorem}\label{theorem_BRW}
Suppose that Assumptions \ref{assumption_RW}, \ref{assumption_supercritical}, \ref{assumption_ZlogZ} and \ref{assumption_Schroeder}  are satisfied. Then, the laws of $ \frac{M_n}{n}$ under $\P^*$ satisfy a large deviation principle with rate function $I^{\textnormal{BRW}}$. More precisely,
\begin{align*}
- I^{\textnormal{BRW}}(x)=
\begin{cases}
\limn \frac{1}{n} \log \P \bigl( \frac{M_n}{n} \geq x \bigr) & \text{for } x \geq x^*,\\
\limn \frac{1}{n} \log \P \bigl( \frac{M_n}{n} \leq x \bigr) & \text{for } x \leq x^*.
\end{cases}
\end{align*}
\end{theorem}

In contrast to the case of independent random walks we only consider the Schröder case (Assumption \ref{assumption_Schroeder}) for the branching random walk.

\begin{remark}
Let us comment on the shape of the rate functions and on Assumption \ref{assumption_Schroeder}.
\begin{itemize}
\item[a)] One can check that the rate function $I^{\textnormal{BRW}}$ is convex. Further note that $I^{\textnormal{ind}}$ is concave on the interval $[0,x^*]$.
\item[(b)] Assumption \ref{assumption_Schroeder} is only needed for the lower deviations ($x<x^*$) in Theorem \ref{theorem_BRW}. In the Böttcher case, i.e.~if Assumption~\ref{assumption_Schroeder} is not satisfied, the strategy for lower deviations is different. We refer to \cite{CH18} for recent results.
\end{itemize}
\end{remark}

For $x > x^*$ we have $I^{\text{BRW}}(x)=I^{\text{ind}}(x)$. In this case the strategy is the same as for independent random walks. The strategy in the case $x < x^*$ goes as follows. At time $tn$ there are only $k^*$ particles, and the position of one of those particles is smaller than its expectation. All other $k^*-1$ particles are killed at time $tn$. Note that by Assumption \ref{assumption_Schroeder} either $k^*=1$ or particles may have no offspring with positive probability. Afterwards, every particle moves and branches according to its usual behaviour.

Further notice, that in contrast to the case of independent random walks, the number of particles can also grow exponentially if $x<0$. It suffices to have a small number of particles at time $tn$ for some $t \in [0,1]$.

\begin{figure}[ht]
   \centering
\begin{tikzpicture}[
    baseline=0pt,
    declare function={I_brw_1(\x)=1.7*(\x)*(\x)+2.5;},
    declare function={I_brw_2(\x)=2.5-0.625*(\x)*(\x);},
    declare function={I_brw_schwarz_1(\x)=0.3*(\x-2)*(\x-2);},
    declare function={I_brw_3(\x)=2*(\x-2)*(\x-2);}
]

\begin{axis}[
    domain=-1:3,
    axis y line =box, 
    axis x line =box, 
    xtick={0, 2},
    xticklabels={$0$, $x^*$},
    ytick={0},
    yticklabels={0},
    xmin=-1,
    xmax=3,
    ymin=0,
    ymax=6,
    x=1.2cm,
    y=0.8cm,
    smooth,
    title = {$I(x^*)= \log m$},
]

\addplot [domain=-1:0,dashed] {I_brw_1(x)} ; \label{figure_I_ind}
\addplot [domain=0:2,dashed] {I_brw_2(x)} ; 
\addplot [domain=-1:2] {I_brw_schwarz_1(x)} ; 
\addplot [domain=2:3] {I_brw_3(x)} ; \label{figure_I_BRW}

\end{axis}
\end{tikzpicture}%
\hskip 20pt
\begin{tikzpicture}[
    baseline=0pt,
    declare function={I_brw_3(\x)=1.3*(\x)*(\x)+2.6;},
    declare function={I_brw_4(\x)=2.6-0.4*(\x)*(\x);},
    declare function={I_brw_schwarz_3(\x)=0.3*(\x-2)*(\x-2);},
    declare function={I_brw_5(\x)=0;},
    declare function={I_brw_6(\x)=5;},
]

\begin{axis}[
    domain=-1:3,
    axis y line =box, 
    axis x line =box, 
    xtick={0, 2},
    xticklabels={$0$, $x^*$},
    ytick={0,1,5},
    yticklabels={0,$I^\textnormal{GW}(I(x^*))$, $\infty$},
    xmin=-1,
    xmax=3,
    ymin=0,
    ymax=6,
    smooth,
    x=1.2cm,
    y=0.8cm,
    enlarge y limits=false,clip=false,
    title = {$I(x^*)< \log m$},
]

\addplot [domain=-1:0,dashed] {I_brw_3(x)} ;
\addplot [domain=-0:2,dashed] {I_brw_4(x)} ;
\addplot [domain=-1:2] {I_brw_schwarz_3(x)} ;
\addplot [domain=2:3] {I_brw_6(x)} ;
\draw[fill] (axis cs: 2,0) circle [radius=0.06cm];
\draw[fill=white] (axis cs: 2,1) circle [radius=0.06cm];
\draw[fill=white] (axis cs: 2,5) circle [radius=0.06cm];

\end{axis}
\end{tikzpicture}

\caption{The figure shows the qualitative behaviour of the rate function of the branching random walk (\ref{figure_I_BRW}) and the rate function of independent random walks (\ref{figure_I_ind}). }
\label{fig_rate_functions}
\end{figure}
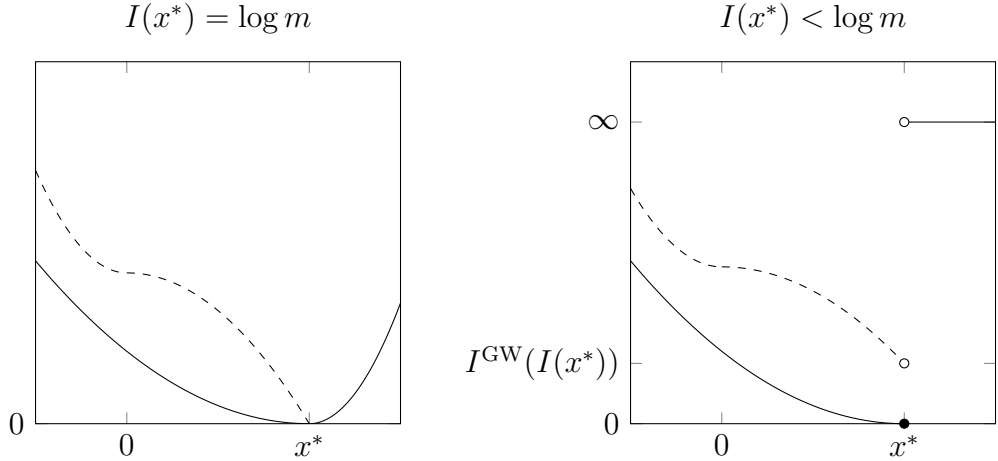

To compare the rate functions, note that the maximum of independent random walks stochastically dominates the maximum of the branching random walk, see Lemma~\ref{lemma_BRW_vs_ind_RW2}. Therefore, $I^{\text{ind}}(x) \leq I^{\text{BRW}}(x)$ for $x>x^*$, respectively $I^{\text{ind}}(x) \geq I^{\text{BRW}}(x)$ for $x<x^*$. For $x<x^*$, the inequality is in general strict. For $x>x^*$, the rate functions coincide, see the argument above.

Let us now comment on the shape of the rate functions. If $I(x)=\infty$ for some $x>x^*$, also $I^{\text{ind}}(x)=I^{\text{BRW}}(x)=\infty$. More precisely, $I(x)=\infty$ already implies $\P(X_1 \geq x- \varepsilon)=0$ for some $\varepsilon>0$ and therefore $M_n \leq x^*n $, respectively $\tilde{M}_n \leq x^*n$ almost surely.

If $I(x^*)= \log m$, then the rate functions $I^{\text{ind}}(x)$ and $I^{\text{BRW}}(x)$ are continuous at $x=x^*$. However, if $I(x^*)< \log m$, the rate functions $I^{\textnormal{ind}}(x)$ and $I^{\textnormal{BRW}}(x)$ are infinite for $x>x^*$, since $I(x)=\infty$. Therefore, they are not continuous from the right at $x=x^*$. The rate function $I^{\textnormal{BRW}}(x)$ is continuous from the left at $x=x^*$, since $I^{\textnormal{BRW}}(x) \leq \rho \bigl(1-\frac{x}{x^*} \bigr)$ for $x<x^*$. However, $I^{\textnormal{ind}}(x)$ is also not continuous from the left at $x=x^*$. In particular, $ \lim_{x\nearrow x^*} I^{\textnormal{ind}}(x) \in (0, \infty)$ if $\rho< \infty$.

An intuitive explanation of this discontinuity is the following. If there are at least $\exp\bigl(I(x^*) n \bigr)$ particles at time $n$, then $\tilde{M}_n = x^*n + o(n)$ with high probability. For a smaller linear term there have to be less particles, hence for all $x<x^*$ the probability $\P^*(\tilde{M}_n \leq xn)$ is bounded from below by the probability to have at most $\exp\bigl(I(x^*) n \bigr)$ particles at time $n$, which decays exponentially. Note that for the branching random walk, in contrast, it suffices to have a small number of particles at the beginning.

\section{Preliminaries} \label{preliminaries}
Before we prove the main results, we collect some preliminaries which are needed throughout the proofs.
As already mentioned in the introduction, conditioned on survival, the linear speed of the branching random walk equals $x^*$, i.e.
\begin{equation}\label{linspeed}
\limn \frac{M_n}{n} = x^* \quad \P^*\text{-a.s.}
\end{equation}

We often use the following simple  inequalities.
\begin{enumerate}
\item[(U1)] We have $1-x \geq e^{-ex}$ for $x \in [0,e^{-1}]$.
\item[(U2)] We have $1-(1-x)^y \geq xy(1-xy)$ for $x \in [0,1]$ and $y \geq 0$.
\end{enumerate}

\begin{proof}
Both inequalities follow after some elementary calculations.
\begin{enumerate}
\item[(U1)] The function $f(x)=1-x-e^{-ex}$ is increasing on $[0,e^{-1}]$. The claim follows, as $f(0)=0$.
\item[(U2)]
The function $g(x)=1-e^{-x}-x(1-x)$ is increasing for $x \geq 0$. As $g(0)=0$, we have $1-e^{-x} \geq x(1-x)$. Using additionally the well known inequality $1-x \leq e^{-x}$, we get
\begin{equation*}
1-(1-x)^y \geq 1- e^{-xy} \geq xy(1-xy). \qedhere
\end{equation*} 
\end{enumerate}
\end{proof}

Furthermore, we need the following simple inequality about the sum of random variables. Note that the random variables in this lemma do not have to be independent.

\begin{lemma}\label{lemma_sum_rv}
For $i \in \N$ let $X_i$ be a Bernoulli($p_i$) random variable with $\inf_{i\in \N}p_i =:p$. Then for every $a >0$ and $n \in \N$ 
\begin{equation*}
\P \left(\frac{1}{n} \sum_{i=1}^n X_i \geq a \right) \geq p-a. 
\end{equation*}
\end{lemma}

A proof can be found in \cite{DGHPW18}.

For $i \in \N$ let $(a_n^i)_{n \in \N}$ be a sequence of positive numbers and $a^i= \limsup_{n \to \infty} \frac{1}{n} \log a_n^i$. Then,
for all $k \in \N$ it holds that
\begin{equation}\label{sevterms}
\limsup_{n \to \infty} \frac{1}{n} \log  \sum_{i=1}^k a_n^i = \max_{i \in \{1, \ldots , k\}} a^i.
\end{equation}
A proof can e.g.~be found in \cite[Lemma 1.2.15]{DZ10}.

We often need to estimate the number of particles at time $n$, which has expectation $m^n$. Let $W_n = \frac{Z_n}{m^n}$ and let $(\mathcal{F}_n)_{n \in \N}$ be the natural filtration of the Galton-Watson process, i.e.~$\mathcal{F}_n= \sigma(Z_1, \ldots, Z_n)$. 
The process $(W_n)_{n \in \N}$ is a martingale with respect to the filtration $(\mathcal{F}_n)_{n \in \N}$. Therefore, $W_n \to W$ almost surely, where $W$ is an almost surely finite random variable. The following well-known theorem shows that under our assumptions, the limit $W$ is non-trivial, i.e.~$\P(W=0)<1$. 
\begin{theorem}[Kesten-Stigum] \label{Lemma_GWP_martingale}
 If Assumption \ref{assumption_supercritical} and Assumption \ref{assumption_ZlogZ} are satisfied, we have 
\begin{equation*}
\E[W]=1 \quad \text{and } \quad \P(W=0)=q < 1.
\end{equation*}
\end{theorem}
A proof can e.g.~be found in \cite[Chapter 1, Section 10, Theorem 1]{AN04}. As already mentioned above, a supercritical Galton-Watson process survives with probability $1-q$. In the proof of Theorem~\ref{theorem_ind_RW} and Theorem~\ref{theorem_BRW} we also need the asymptotics of the survival probability of a critical Galton-Watson process.

\begin{theorem} \label{theorem_survival_critical}
Let $m=1$ and $p(1)<1$. Then, $\limn n\P(Z_n>0)=\frac{2}{\operatorname{Var}(Z_1)}$.
\end{theorem}
A proof can be found in \cite[Chapter 1, Section 9, Theorem 1]{AN04}.

Cramér's theorem gives the exponential decay rate of the probability $\P(S_n \geq xn)$ for $x>0$. The following theorem gives the precise asymptotics for this probability.

\begin{theorem} \label{theorem_RW_precise_asyptotics}
Let $x>0$ and $I(x)< \infty$. There exists an explicit constant $c>0$ such that
\begin{equation*}
\limn \sqrt{n}e^{I(x)n} \P \Bigl( \frac{S_n}{n} \geq x \Bigr) = c.
\end{equation*}
\end{theorem}
A proof can be found in \cite[Theorem 3.7.4]{DZ10}. Furthermore, we need some properties of the rate function $I$.

\begin{lemma} \label{lemma_prop_I}
Assume that there exists $x \in \R$ such that $I(x) < \infty$ and $I(x + \varepsilon) = \infty$ for all $\varepsilon>0$. Then $\P(X_1 >x)=0$ and $\P(X_1=x)=e^{-I(x)}$.
\end{lemma}

\begin{proof} 
Let $x \in \R$ such that $I(x) < \infty$ and $I(x + \varepsilon) = \infty$ for all $\varepsilon>0$. Assume that $\P(X_1 >x)>0$. Then, there exists $\varepsilon>0$ such that $\P(X_1 \geq x+ \varepsilon)>0$. However,
\begin{align} \label{proof_prop_I_1}
I(x+ \varepsilon)&= \sup_{\lambda \in \R} \bigl(\lambda (x+\varepsilon) - \log \E \bigl[e^{\lambda X_1} \bigr]\bigr) \nonumber\\
&\leq \sup_{\lambda \in \R} \bigl(\lambda (x+\varepsilon)- \log \bigl( e^{\lambda (x+\varepsilon)} \P(X_1 \geq x + \varepsilon) \bigr) \nonumber\\
& = - \log \P(X_1 \geq x + \varepsilon) < \infty,
\end{align}
which leads to a contradiction. It remains to show that $\P(X_1=x)=e^{-I(x)}$. Analogously to \eqref{proof_prop_I_1}, we get
\begin{equation*}
I(x) \leq \sup_{\lambda \in \R} \bigl(\lambda x- \log \bigl( e^{\lambda x} \P(X_1 =x) \bigr) = - \log \P(X_1=x).
\end{equation*}
Moreover, since $\P(X_1 > x)=0$ we have for all $\varepsilon>0$
\begin{equation*}
I(x) \geq \sup_{\lambda \in \R} \bigl(\lambda x- \log \bigl( e^{\lambda x} \P(X_1 \in (x-\varepsilon, x] + e^{\lambda (x - \varepsilon)} \bigr).
\end{equation*}
Letting $\lambda \to \infty$ and $\varepsilon \to 0$ shows that
\begin{equation*}
I(x) \geq  - \log \P(X_1=x),
\end{equation*}
which finishes the proof.
\end{proof}

\section{Proofs} 
In this section we prove the main results of the paper.

\subsection{Independent random walks} \label{chapter_brw_proofs_indRW}
\begin{proof}[Proof of Theorem \ref{theorem_ind_RW}]
\textbf{1. Case}: $x > x^*$

Following the strategy explained in Section~3, independence of the random walks and (U2) yields
\begin{align} \label{proof_thm_ind_RW_1}
\P^* \Bigl( \frac{\tilde{M}_n}{n} \geq x \Bigr)
& = \E^* \biggr[ 1- \Bigr(1- \P \Bigl( \frac{S_n}{n} \geq x \Bigr) \Bigr)^{Z_n} \biggr] \nonumber\\
& \geq \P^* \Bigl(Z_n \geq \frac{1}{2} m^n \Bigr) \cdot \E \biggr[ 1- \Bigr(1- \P \Bigl( \frac{S_n}{n} \geq x \Bigr) \Bigr)^{\frac{1}{2} m^n} \biggr] \nonumber \\
& \geq \P^* \Bigl(W_n \geq \frac{1}{2} \Bigr) \P \Bigl( \frac{S_n}{n} \geq x \Bigr) \frac{1}{2} m^n \Bigl( 1- \P \Bigl( \frac{S_n}{n} \geq x \Bigr) \frac{1}{2} m^n \Bigr).
\end{align}
By Cramér's theorem, $\P \bigl( \frac{S_n}{n} \geq x \bigr) \frac{1}{2} m^n \to  0$ as $n \to \infty$, since $\log m < I(x)$. For the first factor on the right hand side of \eqref{proof_thm_ind_RW_1} we have $\liminf_{n \to \infty} \P^* (W_n \geq \frac{1}{2} ) \geq \P^*(W > \frac{1}{2}) >0$, since $\E^*[W]\geq \E[W]=1$ by Theorem~\ref{Lemma_GWP_martingale}. Together with \eqref{sevterms} we conclude 
\begin{equation*}
\P^* \Bigl( \frac{\tilde{M}_n}{n} \geq x \Bigr) \geq \exp \bigl( - (I(x) - \log m )n + o(n) \bigr),
\end{equation*}
which yields the lower bound. For the upper bound, the Markov inequality yields
\begin{align}\label{proof_thm_ind_RW_2}
\P^* \Bigl( \frac{\tilde{M}_n}{n} \geq x \Bigr) = \P^* \Bigl( \sum_{i=1}^{Z_n} \mathds{1}_{\{S_n^i \geq nx\}} \geq 1 \Bigr) \leq \P \Bigl( \frac{S_n}{n} \geq x \Bigr)\E^*[Z_n] = \P \Bigl( \frac{S_n}{n} \geq x \Bigr) \frac{m^n}{1-q},
\end{align}
which immediately implies the claim.

\textbf{2. Case}: $0 < x < x^*$

Since the rate function $I$ is strictly increasing on the interval $[0,x^*]$, we can choose $\varepsilon>0$ such that $\varepsilon < I(x) < \log m - \varepsilon$. We prove the upper bound first. Using the inequality $1-y \leq e^{-y}$ and Theorem~\ref{lemma_LDP_Zn}, we have for $n$ large enough
\begin{align*}
\P^* \Bigl(\frac{\tilde{M}_n}{n} \leq x \Bigr)
& = \E^* \biggl[\Bigl(1-\P \Bigl(\frac{S_n}{n} > x \Bigr)\Bigr)^{Z_n} \biggr] \leq \E^* \biggl[\exp \Bigl(-\P \Bigl(\frac{S_n}{n} > x \Bigr) Z_n \Bigr) \biggr]\\
& \leq \P^* \Bigl(Z_n \leq e^{(I(x)+\varepsilon)n} \Bigr) + \exp \bigr(-e^{\varepsilon n+o(n)} \bigr)\\
&= \exp \Bigl(- \bigl( I^{\text{GW}}(I(x) +\varepsilon \bigr) n + o(n) \Bigr).
\end{align*}
Letting $\varepsilon \to 0$ yields the upper bound. Note that $I^{\text{GW}}$ defined in \eqref{def_I_GW} is continuous. The proof for the lower bound is similar. More precisely, since $\P (\frac{S_n}{n} > x) < e^{-1}$ for $n$ large enough, (U1) yields for $n$ large enough
\begin{align*}
\P^* \Bigl(\frac{\tilde{M}_n}{n} \leq x \Bigr)
& = \E^* \biggl[\Bigl(1-\P \Bigl(\frac{S_n}{n} > x \Bigr)\Bigr)^{Z_n} \biggr] \geq \E^* \biggl[\exp \Bigl(-e \cdot \P \Bigl(\frac{S_n}{n} > x \Bigr) Z_n \Bigr) \biggr]\\
& \geq \P^* \Bigl(Z_n \leq e^{(I(x)-\varepsilon)n} \Bigr) \cdot \exp \bigr(-e^{-\varepsilon n+o(n)}  \bigr)\\
&= \exp \Bigl(- I^{\text{GW}}(I(x) - \varepsilon ) n + o(n) \Bigr).
\end{align*}
Letting $\varepsilon \to 0$ yields the lower bound.

\textbf{3. Case}: $x \leq 0$

We first consider $x <0$. For the upper bound we have for $K \in \N$
\begin{equation} \label{proof_thm_ind_RW_3}
\P^* \Bigl(\frac{\tilde{M}_n}{n} \leq x \Bigr)
= \E^* \biggl[\P \Bigl(\frac{S_n}{n} \leq x \Bigr)^{Z_n} \biggr] \leq \sum_{k=1}^K \P \Bigl(\frac{S_n}{n} \leq x \Bigr)^k \P^*(Z_n=k) + \P \Bigl(\frac{S_n}{n} \leq x \Bigr)^K.
\end{equation}
By Theorem~\ref{lemma_LDP_rho}, the probability $\P(Z_n=k)$ is of order $\exp(-\rho n + o(n))$ for all $k \in A$ (defined in \eqref{def_A}) and $\P(Z_n=k)=0$ otherwise. For all $K \in \N$, \eqref{sevterms} yields 
\begin{equation*}
\limsup_{n \to \infty} \frac{1}{n} \log \P^* \Bigl(\frac{\tilde{M}_n}{n} \leq x \Bigr) 
\leq  \max \bigl\{-(k^*I(x)+\rho), -KI(x) \bigr\}. 
\end{equation*}
Hence, letting $K \to \infty$ proves the upper bound. Note that $I(x)>0$ for $x<0$.  As in the proof of \eqref{proof_thm_ind_RW_3} we have
\begin{align*}
\P^* \Bigl(\frac{\tilde{M}_n}{n} \leq x \Bigr)
& = \E^* \biggl[\P \Bigl(\frac{S_n}{n} \leq x \Bigr)^{Z_n} \biggr] \geq \P \Bigl(\frac{S_n}{n} \leq x \Bigr)^{k^*} \cdot \P(Z_n=k^*)\\
&= \exp \bigl( -(k^*I(x)+\rho)n +o(n) \bigr),
\end{align*}
which shows the lower bound. For $x=0$ the result follows from continuity of the rate function $I^\textnormal{ind}$ at 0.

\textbf{4. Case}: $x=x^*$\\
In the same was as in  \eqref{proof_thm_ind_RW_2},
\begin{equation} \label{proof_thm_ind_RW_4}
\P \Bigl( \frac{\tilde{M}_n}{n} \leq x^* \Bigr) = 1- \P \Bigl( \frac{\tilde{M}_n}{n} > x^* \Bigr) \geq 1- \P \Bigl( \frac{S_n}{n} > x^* \Bigr) \frac{m^n}{1-q}.
\end{equation}
Now we have to distinguish two cases. If $I(x^*) = \log m$, then the right hand side of \eqref{proof_thm_ind_RW_4} converges to 1 as $n \to \infty$ by Theorem~\ref{theorem_RW_precise_asyptotics}. If $I(x^*) < \log m$, then $I(x)=\infty$ for all $x>x^*$ and therefore $\P(X_1>x^*)=0$ by Lemma~\ref{lemma_prop_I}. Hence, the right hand side of \eqref{proof_thm_ind_RW_4} equals 1. In both cases we get
\begin{equation*}
\limn \frac{1}{n} \log \P \Bigl( \frac{\tilde{M}_n}{n} \leq x^* \Bigr)=0.
\end{equation*}
Since $\P (\tilde{M}_n \geq x^* n ) \geq \P (M_n \geq x^* n )$, it remains to show that $\P(M_n \geq x^* n )$ decays slower than exponentially in $n$. If $I(x)< \infty$ for some $x>x^*$, then the rate function $I^{\text{BRW}}(x)$ is continuous from the right at $x=x^*$. Since $I^{\text{BRW}}(x) \to 0$ as $x \searrow x^*$ in this case, the claim follows. Therefore assume that $I(x)= \infty$ for all $x>x^*$. By Lemma~\ref{lemma_prop_I} we have $\P(X_1=x^*)=e^{-I(x^*)}$ in this case. Consider the following embedded process. Every particle with step size smaller than $x^*$ at any time is killed. Therefore, the reproduction mean in every step is $\P(X_1=x^*) m \geq 1$. Let $q_n$ be the extinction probability of this process at time $n$. By Theorem~\ref{theorem_survival_critical}, $q_n$ decays slower than exponentially in $n$. Since $\P (M_n \geq x^* n ) \geq q_n$, the claim follows.
\end{proof}


\subsection{Branching random walk} \label{chapter_brw_proofs_BRW}
Before proving Theorem~\ref{theorem_BRW}, we first show that the maximum of independent random walks stochastically dominates the maximum of the branching random walk.
\begin{lemma} \label{lemma_BRW_vs_ind_RW1}
Let $(X_i)_{i \in \N}$ and $(Y_i)_{i \in \N}$ be independent sequences of (not necessarily independent) random variables. Furthermore, assume that the random variables $Y_i, i \in \N$, have the same distribution. Then we have for all $k \in \N$ and $x \in \R$
\begin{equation*}
\P \Bigl( \max_{i \in \{1, \ldots, k\}} \{X_i+Y_i\} \leq x \Bigr) \leq \P \Bigl( \max_{i \in \{1, \ldots, k\}} \{X_i+Y_1\} \leq x \Bigr).
\end{equation*}
In other words, $\max_{i \in \{1, \ldots, k\}} \{X_i+Y_1\} \preceq  \max_{i \in \{1, \ldots, k\}} \{X_i+Y_i\}$
where we write ``$\preceq$'' for the usual stochastic domination.
\end{lemma}

\begin{proof}
Let $i^*$ be the smallest (random) index such that $X_{i^*}=\max_{i \in \{1, \ldots, k\}} X_i$. We have
\begin{equation*}
\P \Bigl( \max_{i \in \{1, \ldots, k\}} \{X_i+Y_i\} \leq x \Bigr) \leq \P \bigl( X_{i^*}+Y_{i^*} \leq x \bigr) = \P \bigl( X_{i^*}+Y_1 \leq x \bigr). \qedhere
\end{equation*}
\end{proof}
As a consequence we can show that the maximum of independent random walks stochastically dominates the maximum of the branching random walk.

\begin{lemma}\label{lemma_BRW_vs_ind_RW2}
We have $ M_n \preceq \tilde{M}_n$ for all $n \in \N$.
\end{lemma}

\begin{proof}
We prove this lemma by induction over $n$. For $n=1$ the inequality is obviously true. Assume that the inequality holds for some $n \in \N$. Let $(\tilde{M}_n^1)_{n \in \N}, (\tilde{M}_n^2)_{n \in \N}, \ldots$ and $(M_n^1)_{n \in \N}, (M_n^2)_{n \in \N}, \ldots$ be independent copies of $(\tilde{M}_n)_{n \in \N}$ and $(M_n)_{n \in \N}$, respectively.
Furthermore, let $(X_1^{i,j})_{i,j \in \N}$ be a collection of i.i.d.~random variables such that $X_1^{1,1}$ has the same distribution as $X_1$. For $i \in \{1, \ldots, Z_1\}$, denote by $Z_n^i$ the number of descendants of the $i$-th particle of the first generation at time $n+1$. Note that $Z_n^i$ equals $Z_n$ in distribution. Using first the induction hypothesis and then Lemma~\ref{lemma_BRW_vs_ind_RW1},
\begin{align*}
M_{n+1} 
& \stackrel{d}{=} \max_{i \in \{1, \ldots, Z_1\}} \{X^i_1+M_n^i\}\\
& \preceq  \max_{i \in \{1, \ldots, Z_1\}} \{X^i_1+\tilde{M}_n^i\} \\
& \preceq  \max_{i \in \{1, \ldots, Z_1\}} \max_{j \in \{1, \ldots, Z_n^i\}} \{ X_1^{i,j} + \tilde{M}_n^i\} \\
& \stackrel{d}{=} \tilde{M}_{n+1} . \qedhere
\end{align*}

\end{proof}
The statement of Lemma~\ref{lemma_BRW_vs_ind_RW2} is also true with respect to $\P^*$.

\begin{proof}[Proof of Theorem \ref{theorem_BRW}]

\textbf{1. Case}: $x > x^*$

Recall that $I^{\text{BRW}}(x) = I^{\text{ind}}(x)$ for $x \geq x^*$. Therefore, the upper bound immediately follows from Theorem~\ref{theorem_ind_RW} and Lemma~\ref{lemma_BRW_vs_ind_RW2}. It remains to prove the lower bound. Let $\varepsilon>0$ such that $(1-\varepsilon)I(x)> \log m$. Recall that for $v \in D_{\varepsilon n}$ the rightmost descendant of $v$ at time $n$ is denoted by $M^v_{(1-\varepsilon)n}$. By Lemma~\ref{lemma_BRW_vs_ind_RW1},
\begin{align}
\P^* \Bigl(\frac{M_n}{n} \geq x \Bigr)
& = \P^* \Bigl( \max \limits_{v \in D_{\varepsilon n}} \frac{M_{(1-\varepsilon)n}^v - S_v}{(1-\varepsilon)n} + \frac{S_v}{(1-\varepsilon)n} \geq \frac{x}{1-\varepsilon} \Bigr) \notag \\ 
& \geq \P^* \Bigl( \max \limits_{v \in D_{\varepsilon n}} \frac{M_{(1-\varepsilon)n}^v - S_v}{(1-\varepsilon)n} + \frac{S_{\varepsilon n}}{(1-\varepsilon)n} \geq \frac{x}{1-\varepsilon} \Bigr) \notag \\
& \geq \P^* \Bigl(\max \limits_{v \in D_{\varepsilon n}} \frac{M_{(1-\varepsilon)n}^v - S_v}{(1-\varepsilon)n} \geq x \Bigr) \cdot \P \Bigl( \frac{S_{\varepsilon n}}{\varepsilon n} \geq x \Bigr). \label{proof_theorem_BRW_case_x>x*_1}
\end{align}
It remains to estimate the first probability on the right hand side of \eqref{proof_theorem_BRW_case_x>x*_1}. Therefore, let $A_k$ be the set of infinite subtrees in generation $k$, i.e.~$A_k=\{v \in D_k \colon |D_l^v|>0 \ \forall l \in \N \}$. Note that $(M_{(1-\varepsilon)n}^v - S_v)_{v \in D_{\varepsilon n}}$ are independent under $\P^*$ conditioned on $A_{\varepsilon n}$.
We can now use similar estimates as in the proof of Theorem~\ref{theorem_ind_RW}. More precisely, by independence and (U2) we get 
\begin{align} \label{proof_theorem_BRW_case_x>x*_2}
& \quad \ \P^* \Bigl(\max \limits_{v \in D_{\varepsilon n}} \frac{M_{(1-\varepsilon)n}^v - S_v}{(1-\varepsilon)n} \geq x \Bigr) \notag \\
& = \E^* \Bigl[ \P^* \Bigl(\max \limits_{v \in D_{\varepsilon n}} \frac{M_{(1-\varepsilon)n}^v - S_v}{(1-\varepsilon)n} \geq x \bigm\vert A_{\varepsilon n} \Bigr) \Bigr] \nonumber\\
& = \E^* \Bigl[ 1- \Bigl(1- \P^* \Bigl(\frac{M_{(1-\varepsilon)n}}{(1-\varepsilon)n} \geq x \Bigr) \Bigr)^{|A_{\varepsilon n}|} \Bigr] \nonumber\\
& \geq \P^* \Bigl(Z_{\varepsilon n} \geq \frac{1}{2} m^{\varepsilon n } \Bigr) \cdot \P^*\Bigl(|A_{\varepsilon n}| \geq \frac{(1-q)}{2} Z_{\varepsilon n} \bigm\vert Z_{\varepsilon n} \geq \frac{1}{2} m^{\varepsilon n } \Bigr) \nonumber\\
& \quad \cdot \biggl( 1- \Bigl(1- \P^* \Bigl(\frac{M_{(1-\varepsilon)n}}{(1-\varepsilon)n} \geq x \Bigr) \Bigr)^{\frac{1-q}{4} m^{\varepsilon n}} \biggr) \notag\\
& \geq \P^* \Bigl(W_{\varepsilon n} \geq \frac{1}{2} \Bigr) \cdot \P^*\Bigl(\frac{1}{Z_{\varepsilon n}} \sum_{v \in D_{\varepsilon n}} \mathds{1}_{\{ |D_l^v|>0 \ \forall l \in \N\}} \geq \frac{(1-q)}{2} \bigm\vert Z_{\varepsilon n} \geq \frac{1}{2} m^{\varepsilon n } \Bigr) \nonumber\\
& \quad \cdot \P^* \Bigl(\frac{M_{(1-\varepsilon)n}}{(1-\varepsilon)n} \geq x \Bigr) \frac{1-q}{4} m^{\varepsilon n} \cdot \Bigl(1-\P^* \Bigl(\frac{M_{(1-\varepsilon)n}}{(1-\varepsilon)n} \geq x \Bigr) \frac{1-q}{4} m^{\varepsilon n} \Bigr).
\end{align}
The first probability on the right hand side of \eqref{proof_theorem_BRW_case_x>x*_2} can be estimated as in \eqref{proof_thm_ind_RW_1}. It holds that $\liminf_{n \to \infty} \P^* (W_{\varepsilon n} \geq \frac{1}{2} ) \geq \P^* (W > \frac{1}{2} )>0$. The second probability is at least $\frac{1-q}{2}$ by Lemma~\ref{lemma_sum_rv}. Furthermore, as for \eqref{proof_thm_ind_RW_2}, the Markov inequality and the choice of $\varepsilon$ yields  
\begin{align} \label{proof_theorem_BRW_case_x>x*_3}
1-\P^* \Bigl(\frac{M_{(1-\varepsilon)n}}{(1-\varepsilon)n} \geq x \Bigr) \frac{1-q}{4} m^{\varepsilon n}
& \geq 1 - \P \Bigl(\frac{S_{(1-\varepsilon)n}}{(1-\varepsilon)n} \geq x \Bigr) \frac{m^n}{4} \nonumber \\
& = 1- \exp \Bigl(-n \Bigl((1- \varepsilon)I(x) - \log m \Bigr) +o(n) \Bigr) \to 1. \end{align}
Combining \eqref{proof_theorem_BRW_case_x>x*_1}, \eqref{proof_theorem_BRW_case_x>x*_2} and \eqref{proof_theorem_BRW_case_x>x*_3} shows
\begin{equation*}
\liminf_{n \to \infty} \frac{1}{n} \log \P^* \Bigl(\frac{M_n}{n} \geq x \Bigr) \geq -\varepsilon (I(x) -\log m) + (1- \varepsilon) \liminf_{n \to \infty} \frac{1}{(1- \varepsilon) n} \log \P^* \Bigl(\frac{M_{(1- \varepsilon) n}}{(1-\varepsilon) n} \geq x \Bigr).
\end{equation*}
This implies the lower bound.

\textbf{2. Case}: $x < x^*$

Following the strategy explained in Section~3, there are only $k^*$ particles at time $tn$ and the position of one particle is smaller than its expectation. Afterwards, all particles move and branch as usual. For the lower bound let $t \in (0, \min \{1- \frac{x}{x^*},1 \}]$ and fix $\varepsilon >0$. Note that $t \in (0, 1- \frac{x}{x^*}]$ if $ x>0$ and $t \in (0, 1 ]$ if $x \leq 0$. We have
\begin{align} \label{proof_theorem_BRW_case_x<x*_1}
\P^* \Bigl(\frac{M_n}{n} \leq x \Bigr)
& \geq \P^* \Bigl(\frac{M_n}{n} \leq x \bigm\vert Z_{tn}=k^* \Bigr) \cdot \P^* (Z_{tn}=k^* ) \nonumber\\
& \geq q^{k^*-1} \P^* \Bigl(\frac{S_{tn}+M_{(1-t)n}}{n} \leq x \Bigr) \cdot \P^* (Z_{tn}=k^* ) \nonumber \\
& \geq q^{k^*-1} \P^* \Bigl(\frac{M_{(1-t)n}}{(1-t)n} \leq x^* + \varepsilon \Bigr) \cdot \P\Bigl(\frac{S_{tn}}{n} \leq \left(x-(1-t)(x^*+ \varepsilon) \right) \Bigr) \nonumber\\
& \quad \cdot \P^* (Z_{tn}=k^* ).
\end{align}
Since the first probability on the right hand side of \eqref{proof_theorem_BRW_case_x<x*_1} converges to 1 almost surely as $n \to \infty$ by \eqref{linspeed}, we get
\begin{align*}
\P^* \Bigl(\frac{M_n}{n} \leq x \Bigr) \geq \exp \Bigl( -\Bigl[I \Bigl(t^{-1}(x-(1-t)(x^*+ \varepsilon)) \Bigr)+ \rho \Bigr]tn + o(n) \Bigr).
\end{align*}
Letting $\varepsilon \to 0$ and since this inequality holds for all $t \in (0, \min \{1- \frac{x}{x^*},1 \}]$, we conclude
\begin{align*}
\liminf_{n \to \infty} \frac{1}{n} \log \P^* \Bigl(\frac{M_n}{n} \leq x \Bigr) \geq \sup_{t \in (0, \min \{1- \frac{x}{x^*},1 \}]} - H(x) = -\inf_{t \in (0, 1]}  H(x).
\end{align*}
For the upper bound define
\begin{equation*}
T_n = \inf \Bigl\{t \geq 0 \colon Z_{tn} \geq n^3 \Bigr\}
\end{equation*}
and for $\varepsilon_1>0$ introduce the set
\begin{equation*} 
F=F(\varepsilon_1)= \Bigl\{\varepsilon_1, 2 \varepsilon_1, \ldots, \Bigl\lceil \min \Bigl\{\Bigl(1- \frac{x}{x^*} \Bigr), 1 \Bigr\} \varepsilon_1^{-1} \Bigr\rceil \varepsilon_1 \Bigr\}.
\end{equation*}
By the definition of $T_n$ we then have
\begin{align} \label{proof_theorem_BRW_case_x<x*_2}
& \quad \ \P^* \Bigl(\frac{M_n}{n} \leq x \Bigr) \nonumber\\
& \leq \P^* \Bigl(T_n > \min \Bigl\{\Bigl(1- \frac{x}{x^*} \Bigr), 1 \Bigr\} \Bigr) + \sum_{t \in F} \P^* \Bigl(\frac{M_n}{n} \leq x \bigm\vert T_n \in \bigl(t - \varepsilon_1, t \bigr] \Bigr) \P^* \bigl(T_n \in \bigl(t - \varepsilon_1, t \bigr] \bigr) \nonumber \\
& \leq \P^* \bigl(  Z_{(\min \{(1- \frac{x}{x^*} ), 1 \})n} \leq n^3 \bigr)+ \sum_{t \in F} \P^* \Bigl(\frac{M_n}{n} \leq x \bigm\vert T_n \in \bigl(t - \varepsilon_1, t \bigr] \Bigr) \P^* (  Z_{(t- \varepsilon_1) n} \leq n^3).
\end{align}

Let $\varepsilon_2>0$. Recall that $A_{tn}$ is the set of infinite subtrees in generation $tn$. Using Lemma~\ref{lemma_BRW_vs_ind_RW1} and the same estimate as in \eqref{proof_theorem_BRW_case_x>x*_2},
\begin{align} \label{proof_theorem_BRW_case_x<x*_3}
& \quad \ \P^* \Bigl(\frac{M_n}{n} \leq x \bigm\vert T_n \in \bigl(t - \varepsilon_1, t \bigr] \Bigr) \nonumber \\
& \leq  \P^* \biggl(\max_{v \in D_{tn}} \frac{S_{tn}+ M^v_{(1-t)n}}{n} \leq x \Bigm\vert T_n \in \bigl(t - \varepsilon_1, t \bigr] \biggr) \nonumber \\
& \leq \P \Bigl(\frac{S_{tn}}{n} \leq -\bigl((1-t)(x^*- \varepsilon_2)-x \bigr) \Bigr)+ \P^* \Bigl(\frac{M_{(1-t)n}}{(1-t)n} \leq x^* - \varepsilon_2 \Bigr)^n \notag \\
& \quad \ + \P^* \bigl( Z_{tn} \leq n^2 \bigm\vert  T_n \in (t - \varepsilon_1, t] \bigr) + \P^* \bigl( |A_{tn}| \leq n  \bigm\vert Z_{tn}>n^2, T_n \in (t - \varepsilon_1, t] \bigr).
\end{align}
The second probability on the right hand side of \eqref{proof_theorem_BRW_case_x<x*_3} converges to 0 almost surely by \eqref{linspeed}. Hence, the second term in \eqref{proof_theorem_BRW_case_x<x*_3} decays faster than exponentially in $n$. For the third term on the right hand side of \eqref{proof_theorem_BRW_case_x<x*_3},
\begin{align} \label{proof_theorem_BRW_case_x<x*_4}
\P^* \bigl( Z_{tn} \leq n^2 \bigm\vert  T_n \in (t - \varepsilon_1, t] \bigr)
& \leq \P^* \bigl(\exists k \in \N \colon Z_k \leq n^2 \bigm\vert Z_0=n^3 \bigr) \notag\\
& \leq \binom{n^3}{n^2} q^{n^3-n^2} \leq \exp \bigl( (n^3-n^2) \log q + 3n^2 \log n \bigr).
\end{align}
In the second inequality we used the fact that for the event we consider, at most $n^2$ of the initial $n^3$ Galton-Watson trees may survive. Note that every initial particle produces an independent Galton-Watson tree.
Similarly to \eqref{proof_theorem_BRW_case_x<x*_4}, we get for the fourth term on the right hand side of \eqref{proof_theorem_BRW_case_x<x*_3}
\begin{equation} \label{proof_theorem_BRW_case_x<x*_5}
\P^* \bigl( |A_{tn}| \leq n  \bigm\vert Z_{tn}>n^2, T_n \in (t - \varepsilon_1, t] \bigr) \leq \binom{n^2}{n} q^{n^2-n} \leq \exp \bigl( (n^2-n) \log q + 2n \log n \bigr).
\end{equation}

Combining \eqref{proof_theorem_BRW_case_x<x*_2}, \eqref{proof_theorem_BRW_case_x<x*_3}, \eqref{proof_theorem_BRW_case_x<x*_4} and \eqref{proof_theorem_BRW_case_x<x*_5} and letting $\varepsilon_1, \varepsilon_2 \to 0$, we conclude with \eqref{sevterms} after a straightforward calculation
\begin{equation*}
\limsup_{n \to \infty} \frac{1}{n} \log \P^* \Bigl(\frac{M_n}{n} \leq x \Bigr) \leq
 - \inf_{t \in (0,\min \{1- \frac{x}{x^*},1 \}]} \Bigr\{t \rho + tI\Bigl(-\frac{(1-t)x^*-x}{t} \Bigr) \Bigr\} = -H(x).
\end{equation*}
Note that we could take the limit $\varepsilon_2 \to 0$, since $I$ is continuous from the right on $(0,\infty)$.

\textbf{3. Case}: $x = x^*$

The proof is analogous to the proof of Theorem~\ref{theorem_ind_RW}.

\end{proof}

\paragraph{Acknowledgments} We thank the referee for reading carefully the first version of this article, and for pointing out several inaccuracies. Furthermore, we thank Stefan Junk for discussions regarding the proof of Theorem~\ref{theorem_BRW}.

\bibliographystyle{amsplain}
\bibliography{BRWldp_final}

\Addresses

\end{document}